%% file: abelianSC.tex
\documentclass[pdftex,11pt,a4paper]{article}

\usepackage{color}
\usepackage{polynom}
\usepackage{pifont}
\usepackage{amsthm}
\usepackage{mathtools}
\usepackage{amsmath}
\usepackage{amsfonts}
\usepackage{graphicx}
\usepackage{subfig}
\usepackage{algorithmic}
\usepackage{algorithm}
\usepackage{hyperref}
\usepackage[hyperpageref]{backref}
\usepackage{appendix}
\usepackage{geometry}
\usepackage{setspace}

\DeclareGraphicsExtensions{.pdf,.png,.jpg}

\input{setting}

\allowdisplaybreaks

\begin{document}  
\title{A Conjecture on Induced Subgraphs of Cayley Graphs}
\author{Aaron Potechin\thanks{Department of Computer Science, University of Chicago, Chicago, IL. Email: potechin@uchicago.edu} \quad \quad Hing Yin Tsang\thanks{Department of Computer Science, University of Chicago, Chicago, IL. Email: hytsang@cs.uchicago.edu}}
\date{\today}
\maketitle

\abstract{In this paper, we propose the following conjecture which generalizes a theorem proved by Huang \cite{Hua19} in his recent breakthrough proof of the sensitivity conjecture. We conjecture that for any Cayley graph $X = \Gamma(G,S)$ on a group $G$ and any generating set $S$, if $U \subseteq G$ has size $|U| > |G|/2$, then the induced subgraph of $X$ on $U$ has maximum degree at least $\sqrt{|S|/2}$}. Using a recent idea of Alon and Zheng \cite{AZ20}, who proved this conjecture for the special case when $G = \mathbb{Z}_2^n$, we prove that this conjecture is true whenever $G$ is abelian. We also observe that for this conjecture to hold for a graph $X$, some symmetry is required: it is insufficient for $X$ to just be regular and bipartite.

\section{Introduction}
Huang \cite{Hua19} gave a remarkably simple and elegant proof of the Sensitivity Conjecture of Nisan and Szegdy \cite{NS94}. His proof showed that any subset $U$ of the $n$-dimensional Boolean cube of size $|U| > 2^{n-1}$ induces a subgraph with maximum degree at least $\sqrt{n}$, which improves a lower bound of Chung et al. \cite{CFGS88} exponentially and is known to imply the Sensitivity Conjecture \cite{GL92}.

The proof has attracted considerable attention and there are already a number of extensions and generalizations to both the technique and the result itself. Among those, Alon and Zheng \cite{AZ20} proved that Huang's result implies the same phenomenon for not just the Boolean cube, but any Cayley graph of $\Z_2^n$. In this paper, we prove that the same phenomenon holds for any Cayley graph of any abelian group and conjecture that it holds for any Cayley graph.
\subsection{Basic definitions and statement of the conjecture}
We need the following basic definitions. All the groups we consider in this paper are finite.
\begin{Def}[Cayley Graphs]
Given a group $G$ and a set of non-identity elements $S$ of $G$, the Cayley graph $X = \Gamma(G,S)$ is the graph with vertices $V(X) = G$ and edges $E(X) = \{(g,sg): g \in G, s \in S\}$. Here we take $X = \Gamma(G,S)$ to be undirected and we assume that $S$ is symmetric, i.e. if $s \in S$ then $s^{-1} \in S$.
\end{Def}
\begin{remark}
Without loss of generality, we can assume that $S$ generates $G$ as otherwise, letting $G'$ be the subgroup of $G$ which is generated by $S$, $\Gamma(G,S)$ consists of $\frac{|G|}{|G'|}$ disjoint copies of $\Gamma(G',S)$.
\end{remark}
\begin{Def}[Boolean Hypercube]
We define $Q_n$ to be the $n$ dimensional boolean hypercube, i.e. $Q_n = \Gamma(Z_2^n,\{e_i: i \in [n]\})$.
\end{Def}
\begin{Def}[Induced Subgraphs]
Given a graph $X$ and a subset of vertices $U \subseteq V(X)$, we define $X(U)$ to be the induced subgraph of $X$ on $U$, i.e. $V(X(U)) = U$ and $E(X(U)) = \{(u_1,u_2) \in E(X): u_1,u_2 \in U\}$.
\end{Def}
With these definitions, we now state our conjecture.
\begin{Conj}\label{nonabelian}
For any Cayley graph $X = (G,S)$ and any $U \subseteq G$ such that $|U| > \frac{|G|}{2}$, the induced subgraph $X(U)$ of $X$ on $U$ has maximum degree at least $\sqrt{|S|/2}$.
\end{Conj}
\section{Proof of the conjecture for abelian groups}
In this section, we prove that \conjref{nonabelian} is true for abelian groups. More preceisely, we prove the following theorem.
\begin{Thm}
\label{abelian}
For any Cayley graph $X = \Gamma(G, S)$ such that $G$ is abelian and any $U \subseteq G$ of size $|U| > |G|/2$, the induced subgraph $X(U)$ of $X$ on $U$ has maximum degree at least $\sqrt{(|S|+t)/2}$ where $t$ is the number of elements in $S$ of order 2.
\end{Thm}
We prove this theorem in two steps:
\begin{enumerate}
\item We show that Huang's theorem \cite{Hua19} implies that the same property holds for products of cycles.
\item We generalize the argument used by Alon and Zheng \cite{AZ20} to prove \conjref{nonabelian} for $G = \mathbb{Z}_2^n$ to prove the result for all abelian $G$.
\end{enumerate}
\subsection{From the Boolean hypercube to products of cycles}
We recall Huang's theorem and show how it implies the same property for products of cycles.

\begin{Thm}[\cite{Hua19}]  
For every integer $n \ge 1$, if $H$ is an induced subgraph of $Q_n$ with at least $(2^{n-1}+1)$ vertices, then the maximum degree of $H$ is at least $\sqrt{n}$.
\end{Thm}
\begin{Cor}
\label{B2pC}
Let $G = \Z_{m_1} \times \cdots \times \Z_{m_d}$, $S = \{\pm{e_1}, \dots, \pm{e_d}\}$, and $X = \Gamma(G, S)$. For any $U \subseteq G$ of size $|U| > |G|/2$, there is an element $u \in U$ and $k \geq \sqrt{d}$ distinct indices $i_1,\ldots,i_k \in [d]$ such that for all $j \in [k]$, either $u + e_{i_j} \in U$ or $u - e_{i_j} \in U$.
\end{Cor}
\begin{proof}
To prove this, we cover $X$ with copies of the Boolean hypercube $Q_d$.
\begin{Def}
Let $U_r = \{r + \sum_{i \in T} e_i:  T \subseteq [d]\}$.
\end{Def}
Observe that $\E_r|U_r \cap U| > 2^{n-1}$ where the expectation is over uniform random $r \in G$. Thus, there must be some $g \in G$ that satisfies $|U_g \cap U| > 2^{n-1}$. Since $X(U_g)$ is isomorphic to the Boolean cube $Q_d$ of dimension $d$, by Huang's theorem, the induced subgraph $X(U_g \cap U)$ of $X$ on $U_g \cap U$ has maximum degree at least $\sqrt{d}$.
\end{proof}

\begin{figure}[H]
\label{coveringargument}
  \centering
    \includegraphics[width=0.2\textwidth]{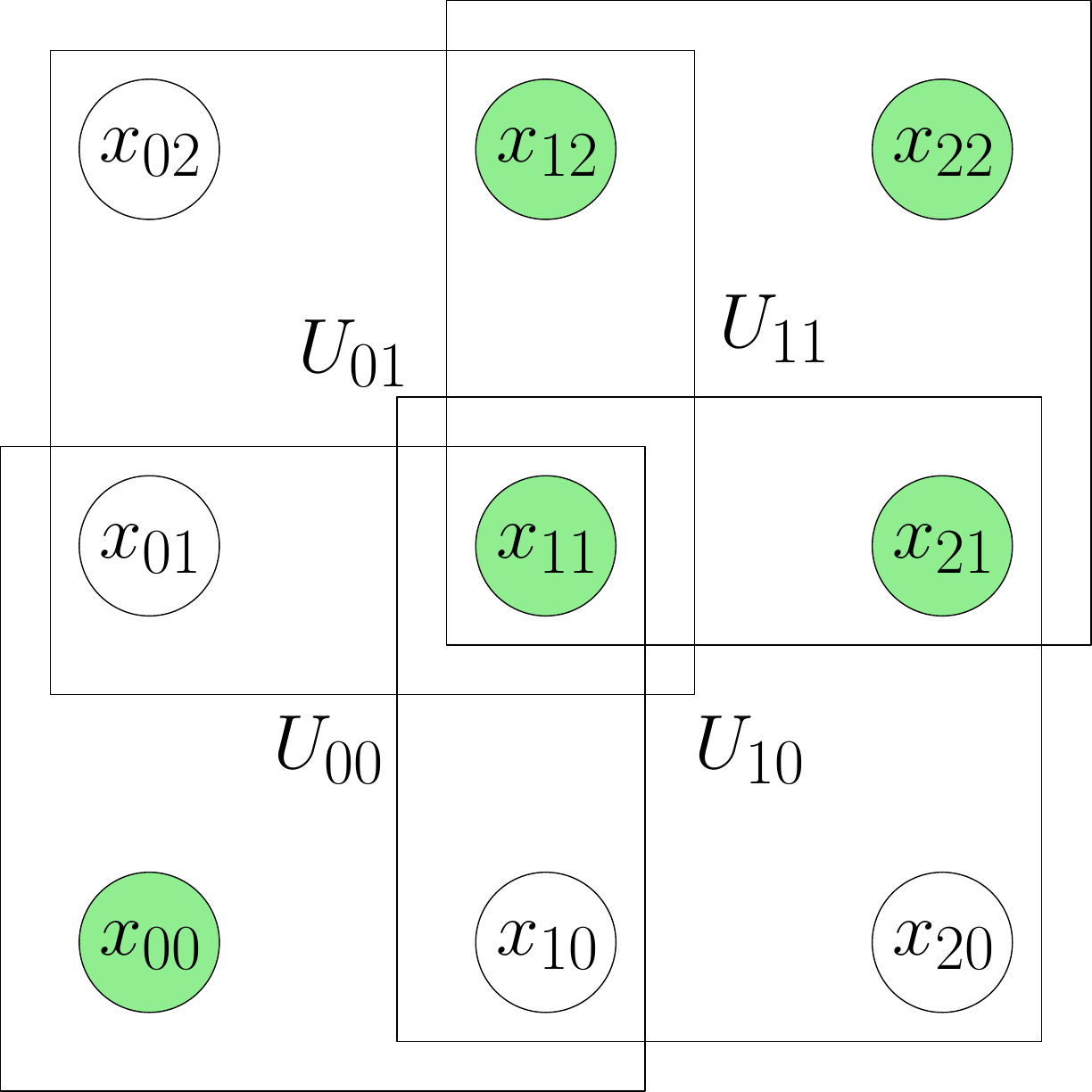}
      \caption{An illustration of some of the sets $\{U_r: r \in G\}$ when $d = 2$ and $m_1 = m_2 = 3$. The sets $U_{02},U_{12},U_{20},U_{21},U_{22}$ wrap around and are not shown.}
\end{figure}
\subsection{From products of cycles to abelian Cayley graphs}
We now apply an argument of Alon and Zheng \cite{AZ20} to prove \thmref{abelian}.

\begin{proof}[Proof of \thmref{abelian}]
We can assume without loss of generality that $G = \Z_{m_1} \times \cdots \times \Z_{m_k}$. Denote $S = \{s_1, \dots, s_t, s_{t+1} \dots, s_d, -s_{t+1}, \dots, -s_{d}\}$ and let $m = lcm(m_1, \dots, m_k)$. We consider the Cayley graph $X' = \Gamma(\Z_m^d, T)$ where $T = \{\pm{e_1}, \dots, \pm{e_d}\}$. Let $A : \Z_m^d \to G$ be a linear map defined by $A(e_i) = s_i$. Note that $A$ is well-defined because $\text{ord}(s_i)|m$ for all $i \in [d]$.

Since $S$ is a generating set of $G$, the linear map $A$ is onto. Thus, for all $g \in G$, $A^{-1}(g)$ has size $m^d/|G|$. It follows that $A^{-1}(U)$ has size $|A^{-1}(U)| = (m^d/|G|)|U| > (m^d/|G|)(|G|/2) = m^d/2$. By \corref{B2pC}, there is a vertex $h \in A^{-1}(U)$ and $k \geq \sqrt{d}$ distinct indices $i_1,\ldots,i_k \in [d]$ such that for all $j \in [k]$, either $h + e_{i_j}$ or $h - e_{i_j}$ is in $A^{-1}(U)$. Take $h_j \in \{h + e_{i_j}, h - e_{i_j}\}$ so that $h_j \in U$ (if both elements are in $U$ then this choice is arbitrary) and observe that for all $j' \neq j \in [k]$, 
\[
A(h_{j'}) - A(h_j) = A(h) \pm A(e_{i_{j'}}) - A(h) \mp A({e_{i_j}}) = \pm{s_{i_{j'}}} \mp{s_{i_j}} \neq 0.
\]
Thus, all $A(h_1), \dots, A(h_k)$ are distinct, contained in $U$, and adjacent to $A(h) \in U$ in $X(U)$ where $X = \Gamma(G,S)$. Finally, $|S| = t+2(d-t)$ and hence $d = (|S|+t)/2$, as desired.
\end{proof}

\section{A counterexample for regular, bipartite graphs} 
In this section, we observe that for \conjref{nonabelian} to hold, it is not sufficient for $X$ to be regular and bipartite. In particular, we construct a regular bipartite graph $X = (L, R, E)$ such that there exists a subset of size $|L|+1$ that induces a subgraph with maximum degree 1.

\begin{figure}[H]
\label{badexample}
  \centering
    \includegraphics[width=0.2\textwidth]{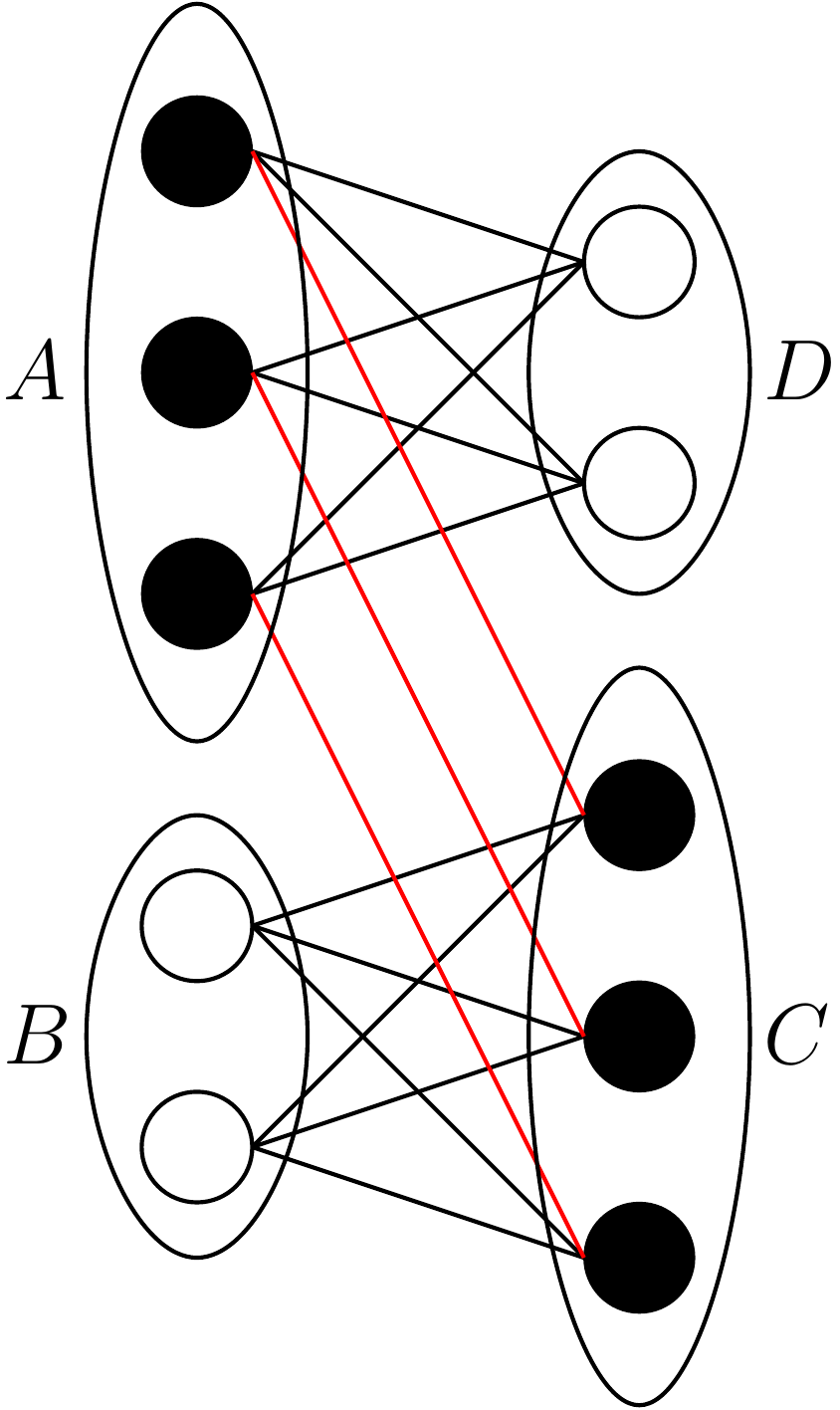}
      \caption{An illustration of the graph $X$ for $n = 2$.}
\end{figure}

Let $A, B, C, D$ be disjoint sets of size $|A| = |C| = n+1$ and $|B| = |D| = n$. Let $L = A \cup B$ and $R = C \cup D$. Let $E$ be the union of a perfect matching between $A$ and $C$, the set of all edges between $A$ and $D$, and the set of all edges between $B$ and $C$. It is straightforward to check that $X$ is $(n+1)$-regular, but the set $A \cup B$ has size $2(n+1) = |L|+1$ and induces a subgraph with maximum degree 1. A concrete drawing of such graph for $n = 2$ is shown in Figure 2.

\section{Open problems}
The obvious open problem is to prove or disprove \conjref{nonabelian}. There are a few more problems that we find interesting:
\begin{enumerate}
\item
Our counterexample shows that we cannot replace being Cayley by just regular in \conjref{nonabelian}. What about vertex-transitive? Can we find a counterexample where $X$ is vertex-transitive or find evidence that being vertex-transitive is sufficient?

\item
Huang actually proved a stronger claim that the Boolean cube $Q_n$ admits an \emph{orthogonal signing}, which is a signed adjacency matrix of $Q_n$ whose eigenvalues are either $-\sqrt{n}$ or $\sqrt{n}$. Alon and Zheng \cite{AZ20} considered the more general \emph{unitary signing} and showed that any Cayley graph of $\Z_2^n$ with respect to $S$ of size at most $n+1$ admits such a signing. They also observed that some Cayley graphs of $\Z_2^n$ of degree $2n+1$ do not admit such signings, which implies that the tight maximum degree bound of induced subgraph cannot always be proved by finding a good signing. Despite that, finding a signing such that all eigenvalues have large modulus still seems to be an interesting problem. For Cartesian products of $d$ many even-length cycles, it is not difficult to find a signing with all eigenvalues at least $\Omega_k(\sqrt{d})$ in absolute values, where $k$ is the maximum length of the cycles in the product. Can we find a signing for any finite bipartite abelian Cayley graph $\Gamma(G, S)$ with all eigenvalues at least $|S|^{\Omega(1)}$ in absolute value?

\end{enumerate}

\section*{Acknowledgement}
We would like to thank Andrew Drucker for helpful conversations.

\bibliography{reference}

\end{document}

%% file: setting.tex
\usepackage{dsfont}

\def\E{{\mathbb E}}

\def\Z{{\mathbb Z}}

\usepackage{hyperref}
\hypersetup{
    colorlinks=true,       
    linkcolor=blue,          
    citecolor=red,        
    filecolor=green,      
    urlcolor=magenta
}

\theoremstyle{theorem}
\newtheorem{Thm}{Theorem}[section]

\newtheorem{Cor}[Thm]{Corollary}

\newtheorem{Conj}{Conjecture}

\theoremstyle{definition}
\newtheorem{Def}[Thm]{Definition}

\theoremstyle{remark}
\newtheorem*{remark}{Remark}

\newcommand{\conjref}[1]{\hyperref[#1]{Conjecture~\ref*{#1}}}
\newcommand{\thmref}[1]{\hyperref[#1]{Theorem~\ref*{#1}}}
\newcommand{\lemmaref}[1]{\hyperref[#1]{Lemma~\ref*{#1}}}
\newcommand{\corref}[1]{\hyperref[#1]{Corollary~\ref*{#1}}}
\newcommand{\propref}[1]{\hyperref[#1]{Proposition~\ref*{#1}}}
\newcommand{\secref}[1]{\hyperref[#1]{Section~\ref*{#1}}}
\newcommand{\refref}[1]{\hyperref[#1]{[\ref*{#1}]}}
\newcommand{\egref}[1]{\hyperref[#1]{Example~\ref*{#1}}}
\newcommand{\problemref}[1]{\hyperref[#1]{Problem~(\ref*{#1})}}
\newcommand{\figref}[1]{\hyperref[#1]{Figure~\ref*{#1}}}

\allowdisplaybreaks